\documentclass{amsart}
\usepackage{amsmath}
\usepackage{amsfonts}
\usepackage{amsthm}
\usepackage{amssymb}
\usepackage[dvipdfm]{graphicx}
\usepackage{amscd}

\title{The asymptotic behavior of Teichm\"uller rays}
\author{Masanori Amano}
\address{Department of Mathematics, Tokyo Institute of Technology, 2-12-1 Ookayama, Meguroku, Tokyo 152-8551, JAPAN}
\email{amano.m.ab@m.titech.ac.jp}
\subjclass[2000]{Primary~32G15, Secondary~30F60}
\keywords{Teichm\"uller space; Teichm\"uller distance; Teichm\"uller geodesic; augmented Teichm\"uller space}

\theoremstyle{plain}
\newtheorem{thm}{Theorem}[section]

\theoremstyle{definition}

\theoremstyle{plain}
\newtheorem{prop}[thm]{Proposition}

\theoremstyle{plain}
\newtheorem{lemma}[thm]{Lemma}

\theoremstyle{plain}
\newtheorem{cor}[thm]{Corollary}

\theoremstyle{remark}

\theoremstyle{remark}
\newtheorem*{rem}{\bf Remark}

\theoremstyle{plain}
\newtheorem{thm*}{Theorem}[section]

\theoremstyle{plain}
\newtheorem{prop*}[thm*]{Proposition}

\theoremstyle{plain}
\newtheorem{cor*}[thm*]{Corollary}

\DeclareMathOperator{\jac}{Jac}

\begin{document}

\begin{abstract}
In this paper, we consider the asymptotic behavior of two Teichm\"uller geodesic rays determined by Jenkins-Strebel differentials, and we obtain a generalization of a theorem in \cite{Amano14}.
We also consider the infimum of the asymptotic distance in shifting base points of the rays along the geodesics.
We show that the infimum is represented by two quantities.
One is the detour metric between the end points of the rays on the Gardiner-Masur boundary of the Teichm\"uller space, and the other is the Teichm\"uller distance between the end points of the rays on the augmented Teichm\"uller space.
\end{abstract}

\maketitle

\section{Introduction} \label{Intro}
Let $X$ be a Riemann surface of genus $g$ with $n$ punctures such that $3g-3+n>0$, and $T(X)$ be the Teichm\"uller space of $X$.
Any Teichm\"uller geodesic ray on $T(X)$ is determined by a holomorphic quadratic differential on an starting point of the ray.
A geodesic ray is called a Jenkins-Strebel ray if it is given by a Jenkins-Strebel differential.
In \cite{Amano14}, we obtain a condition for two Jenkins-Strebel rays to be asymptotic (Corollary 1.2 in \cite{Amano14}).
To obtain this condition, we use Theorem 1.1 in \cite{Amano14} which gives the explicit asymptotic value of the Teichm\"uller distance between two similar Jenkins-Strebel rays with the same end point in the augmented Teichm\"uller space.
In this paper, we improve this theorem, and obtain the asymptotic value of the distance between any two Jenkins-Strebel rays.

Let $r$, $r'$ be Jenkins-Strebel rays on $T(X)$ from $r(0)=[Y,f]$, $r(0)'=[Y',f']$ determined by Jenkins-Strebel differentials $q$, $q'$ with unit norm on $Y$, $Y'$ respectively.
It is known (cf.\cite{HerSch07}) that the Jenkins-Strebel rays $r$, $r'$ have limits, say $r(\infty )$, $r'(\infty )$, on the boundary of the augmented Teichm\"uller space $\hat T(X)$.
Suppose that $r$, $r'$ are similar, that is, there exist mutually disjoint simple closed curves $\gamma _1,\cdots ,\gamma _k$ on $X$ such that the set of homotopy classes of core curves of the annuli corresponding to $q$, $q'$ are represented by $f(\gamma _1),\cdots ,f(\gamma _k)$ on $Y$ and $f'(\gamma _1),\cdots ,f'(\gamma _k)$ on $Y'$ respectively.
We denote by $m_j$, $m'_j$ the moduli of the annuli on $Y$, $Y'$ with core curves homotopic to $f(\gamma _j)$, $f'(\gamma _j)$ respectively.
We can define the Teichm\"uller distance $d_{\hat T(X)}(r(\infty ),r'(\infty ))$ between the end points $r(\infty )$, $r'(\infty )$.

Our main result is the following:
\begin{thm} \label{main}
For any two Jenkins-Strebel rays $r$, $r'$,
	\begin{eqnarray}
	&&\lim _{t\rightarrow \infty }d_{T(X)}(r(t),r'(t))=\label{eq}\\
	&&\left\{\nonumber
	\begin{array}{l}
	\max \left\{\displaystyle \frac{1}{2}\log \max _{j=1,\cdots ,k}\left\{\frac{m'_j}{m_j},\frac{m_j}{m'_j}\right\}, d_{\hat T(X)}(r(\infty ),r'(\infty ))\right\}\\
	\hfill (${\rm if }$r$, $r'$ {\rm are similar}$)\\
	+\infty \ (${\rm otherwise}$) 
	\end{array}
	\right.
	\end{eqnarray}
\end{thm}

\begin{cor} \label{cor}
If $r$, $r'$ are similar, the minimum value of the equation {\rm (\ref{eq})} when we shift the starting points of $r$, $r'$ is given by
	\begin{eqnarray}
	\max \left\{\frac{1}{2}\delta , d_{\hat T(X)}(r(\infty ),r'(\infty ))\right\},\nonumber
	\end{eqnarray}
where $\displaystyle \delta =\frac{1}{2}\log \max _{j=1,\cdots ,k}\frac{m'_j}{m_j}+\frac{1}{2}\log \max _{j=1,\cdots ,k}\frac{m_j}{m'_j}$.
\end{cor}
\begin{rem}
The quantity $\delta $ is known as the {\it detour metric} between end points of the rays $r$, $r'$ in the Gardiner-Masur boundary of $T(X)$.
We refer to \cite{Walsh12}, and also \cite{Amano14}.
\end{rem}

\section{Preliminaries}
\subsection{Teichm\"uller spaces}
Let $X$ be an analytically finite Riemann surface which has genus $g$ and $n$ punctures, briefly, we say it is of type $(g,n)$.
We assume that $3g-3+n>0$.
Let $T(X)$ be the {\it Teichm\"uller space} of $X$.
It is the set of equivalence classes of pairs of a Riemann surface $Y$ and a quasiconformal mapping $f:X\rightarrow Y$.
Two pairs $(Y,f)$ and $(Y',f')$ are equivalent if there is a conformal mapping $h:Y\rightarrow Y'$ such that $h\circ f$ is homotopic to $f'$.
We denote by $[Y,f]$ the equivalence class of a pair $(Y,f)$.
The {\it Teichm\"uller distance} $d_{T(X)}$ is a complete distance on $T(X)$ which is defined by the following.
For any $[Y,f]$, $[Y',f']$ in $T(X)$, $d_{T(X)}([Y,f],[Y',f'])=\frac{1}{2}\log \inf K(h)$, where the infimum ranges over all quasiconformal mappings $h:Y\rightarrow Y'$ such that $h\circ f$ is homotopic to $f'$, and $K(h)$ is the maximal quasiconformal dilatation of $h$.

\subsection{Holomorphic quadratic differentials}
A {\it holomorphic quadratic differential} $q$ on $X$ is a tensor of the form $q(z)dz^2$ where $q(z)$ is a holomorphic function of a local coordinate $z$ on $X$.
For any $q\not =0$, a zero of $q$ or a puncture of $X$ is called a {\it critical point} of $q$.
Then, $q$ has finitely many critical points.
We allow $q$ to have poles of order $1$ at punctures of $X$.
Then, the $L^1$-norm $\|q\|=\iint _{X}|q|$ is finite, where $|q|=|q(z)|dxdy$.
If $\|q\|=1$, we call $q$ of unit norm.

Let $p_0$ be a non-critical point of $q$ and $U$ be a small neighborhood of $p_0$ which does not contain any other critical points of $q$.
For any point $p$ in $U$, we can define a new local coordinate $\zeta (p)=\int _{z(p_0)}^{z(p)}q(z)^{\frac{1}{2}}dz$ on $X$ where $z$ is a local coordinate on $U$.
The coordinate $\zeta $ is called a {\it $q$-coordinate}.
By $q$-coordinates, we see that $q=d\zeta ^2$, and in a common neighborhood of two $q$-coordinates $\zeta _1$, $\zeta _2$, the equation $\zeta _2=\pm \zeta _1+$constant holds.

Suppose that $p_0$ is a critical point of $q$, and its order is $n\geq -1$.
In a small neighborhood of $p_0$ which does not contain any other critical points of $q$, there exists a local coordinate $z$ on $X$ such that $z(p_0)=0$ and $q=z^ndz^2$.
For instance, we refer to \cite{Strebel84}.
For any non-critical point in the neighborhood of $p_0$, there exists a $q$-coordinate $\zeta $.
By $d\zeta ^2=z^ndz^2$, the transformation $\zeta =\frac{2}{n+2}z^{\frac{n+2}{2}}$ holds.
For any $k=0,\cdots ,n+1$, the set $\{ \frac{2\pi k}{n+2}\leq \arg z\leq \frac{2\pi (k+1)}{n+2}\}$ on the $z$-plane is mapped to the half plane $\{ 0\leq \arg \zeta \leq \pi \}$ or $\{ \pi \leq \arg \zeta \leq 2\pi \}$ on the $\zeta $-plane.
We can see the trajectory flow in the neighborhood of $p_0$ as the $n+2$ copies of the half plane with the gluing along each horizontal edge of the planes (Figure \ref{nbd0}).
	\begin{figure}[!ht]
	\begin{center}
	\includegraphics[keepaspectratio, scale=0.6]
	{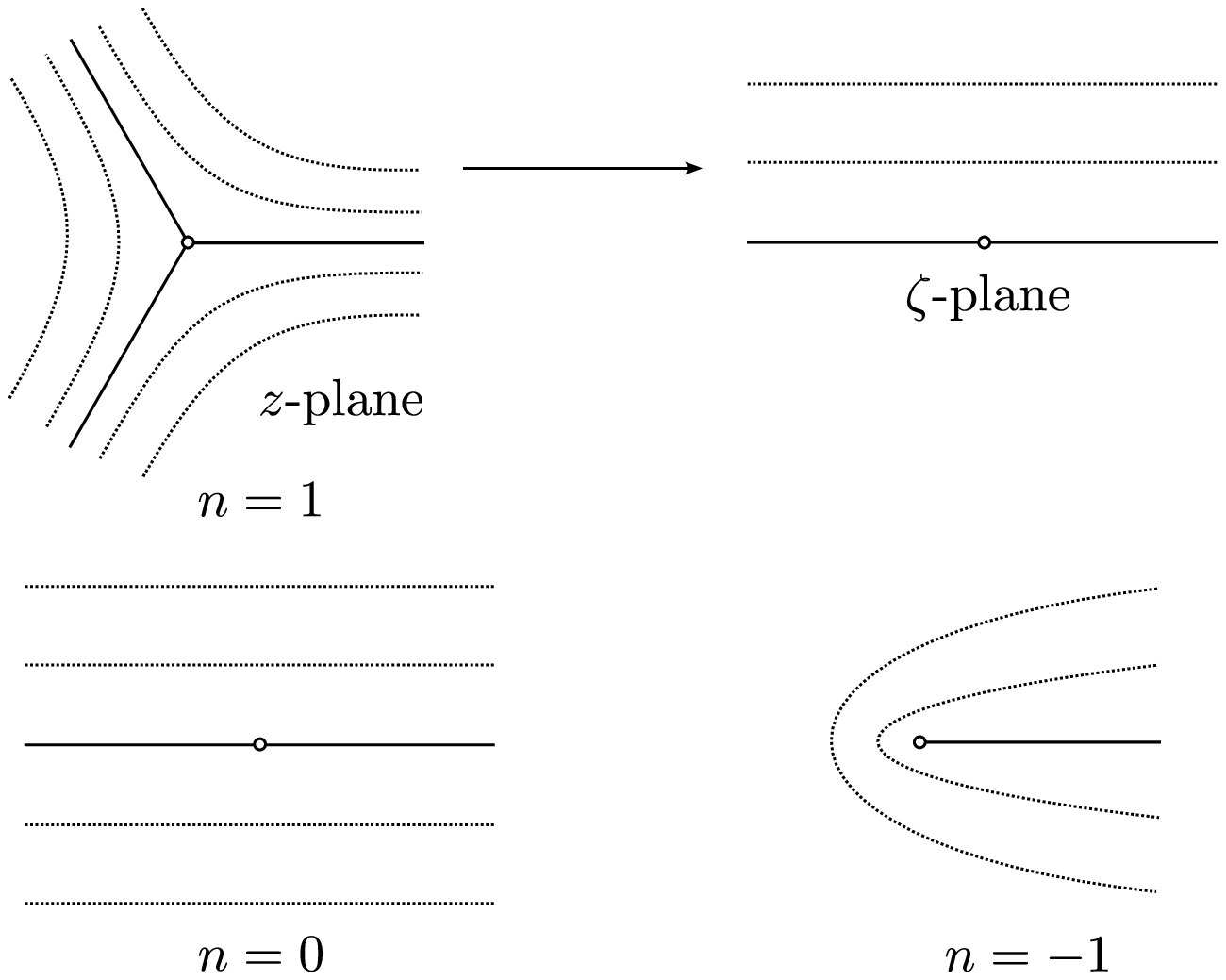}
	\caption{The trajectory flow in the neighborhood of $p_0$ in the case of $n=1,0,-1$}
	\label{nbd0}
	\end{center}
	\end{figure}

A {\it horizontal trajectory} of $q$ is a maximal smooth arc $z=\gamma (t)$ on $X$ which satisfies $q(\gamma (t))(\frac{d\gamma (t)}{dt})^2>0$.
By definition, horizontal trajectories of $q$ does not contain critical points of $q$.
All horizontal trajectories of $q$ are Euclidean horizontal arcs in $q$-coordinates, moreover, by the form of transformations of $q$-coordinates, ``horizontal directions" are preserved.
A {\it saddle connection} of $q$ is a horizontal trajectory which joins critical points of $q$.
We denote by $\Gamma _q$ the set of all critical points of $q$ and all saddle connections of $q$.
Any component of $X-\Gamma _q$ is classified to the following two cases.
	\begin{itemize}
	\item {\it Annulus}: It is an annulus which is swept out by simple closed horizontal trajectories of $q$.
	These are free homotopic to each other.
	We call the simple closed horizontal trajectories the {\it core curves} of the annulus.
	\item {\it Minimal domain}: This domain is generated by infinitely many recurrent horizontal trajectories which are dense in the domain.
	\end{itemize}
Since $q$ has finitely many critical points, the number of components of $X-\Gamma _q$ is finite.
If $X-\Gamma _q$ has only annuli, we call $q$ a {\it Jenkins-Strebel differential}.

\subsection{Teichm\"uller geodesic rays} \label{geodesic}
For any holomorphic quadratic differential $q\not =0$ on $X$, a quasiconformal mapping $f:X\rightarrow Y$ whose Beltrami coefficient is of the form $\mu _f=-\frac{K(f)-1}{K(f)+1}\frac{\bar q}{|q|}$ is called the {\it Teichm\"uller mapping}.
For any quasiconformal mapping $g:X\rightarrow Y$, there exists a Teichm\"uller mapping $f:X\rightarrow Y$ which is homotopic to $g$.
Furthermore, the Teichm\"uller mapping satisfies $K(f)\leq K(g)$ where the equality holds if and only if $f=g$.
\begin{rem}
More generally, If $X$ and $Y$ have same genus and punctures, for any orientation preserving homeomorphism $g:X\rightarrow Y$, there exists a Teichm\"uller mapping $f:X\rightarrow Y$ which is homotopic to $g$ (Theorem 1 in \S 1.5 of Chapter II of \cite{Abikoff80}).
\end{rem}
Let $f:X\rightarrow Y$ be a Teichm\"uller mapping and $q$ be the associated unit norm holomorphic quadratic differential on $X$.
In this situation, there exists a unit norm holomorphic quadratic differential $\varphi $ on $Y$ such that $f$ maps each zero of order $n$ of $q$ to a zero of order $n$ of $\varphi $, and is represented by $w\circ f\circ z^{-1}(z)=K(f)^{-\frac{1}{2}}x+iK(f)^{\frac{1}{2}}y$ where $z=x+iy$ and $w$ are $q$ and $\varphi $-coordinates respectively.
Such $\varphi $ is uniquely determined.
For more details of the discussion, we refer the reader to \cite{ImaTan92}.

Let $p=[Y,f]$, $q\not =0$ be a unit norm holomorphic quadratic differential on $Y$, and $z$ be any $q$-coordinate.
The mapping $r:\mathbb R_{\geq 0}\rightarrow T(X)$ is called a {\it Teichm\"uller geodesic ray from $p$ determined by $q$} if for any $t\geq 0$, we assign a point $[Y_t,g_t\circ f]$ in $T(X)$ to $r(t)$ where $g_t$ is a Teichm\"uller mapping on $Y$ which is of the form $z=x+iy\mapsto z_t=e^{-t}x+ie^ty$, and $Y_t$ is a Riemann surface which is determined by the coordinates $z_t$.
We assume that $g_0=id_Y$ and $Y_0=Y$.
By properties of Teichm\"uller mappings, we have $d_{T(X)}(r(s),r(t))=|s-t|$ for any $s,t\geq 0$.
If $q$ is Jenkins-Strebel, we call $r$ a {\it Jenkins-Strebel ray}.

Let $r$, $r'$ be any two Jenkins-Strebel rays on $T(X)$ from $r(0)=[Y,f]$, $r'(0)=[Y',f']$ determined by Jenkins-Strebel differentials $q$, $q'$ with unit norm on $Y$, $Y'$ respectively.
The rays $r$, $r'$ are {\it similar} if there exist mutually disjoint simple closed curves $\gamma _1,\cdots ,\gamma _k$ on $X$ such that the set of homotopy classes of core curves of the annuli corresponding to $q$, $q'$ are represented by $f(\gamma _1),\cdots ,f(\gamma _k)$ on $Y$ and $f'(\gamma _1),\cdots ,f'(\gamma _k)$ on $Y'$ respectively.

\subsection{Augmented Teichm\"uller spaces}
We refer to \cite{HerSch07} and \cite{ImaTan92} for augmented Teichm\"uller spaces.
Let $R$ be a connected Hausdorff space which satisfies following conditions:
	\begin{itemize}
	\item Any $p\in R$ has a neighborhood which is homeomorphic to the unit disk $\mathbb D=\{|z|<1\}$ or the set $\{(z_1,z_2)\in {\mathbb C}^2 |\ |z_1|<1, |z_2|<1, z_1\cdot z_2=0\}$. (In the latter case, $p$ is called a {\it node} of $R$.)
	\item Let $p_1,\cdots ,p_k$ be nodes of $R$.
	We denote by $R_1,\cdots ,R_r$ the connected components of $R-\{p_1,\cdots ,p_k\}$.
	For any $i=1,\cdots ,r$, each $R_i$ is a Riemann surface of type $(g_i,n_i)$ which satisfies $2g_i-2+n_i>0$, $n=\sum _{i=1}^rn_i-2k$ and $g=\sum _{i=1}^rg_i-r+k+1$.
	\end{itemize}
We call $R$ the {\it Riemann surface of type $(g,n)$ with nodes}.

The {\it augmented Teichm\"uller space} $\hat T(X)$ is the set of equivalence classes of pairs of a Riemann surface of type $(g,n)$ with or without nodes $R$ and a deformation $f:X\rightarrow R$.
The deformation $f$ is a continuous mapping such that some disjoint loops on $X$ are contracted to nodes of $R$, and is homeomorphic except to these loops. 
Two pairs $(R,f)$ and $(R',f')$ are equivalent if there is a conformal mapping $h:R\rightarrow R'$ such that $h\circ f$ is homotopic to $f'$, where the conformal mapping means that each restricted mapping of a component of $R-\{$nodes of $R\}$ onto a component of $R'-\{$nodes of $R'\}$ is conformal.
Obviously, $T(X)$ is included in $\hat T(X)$.
A topology of $\hat T(X)$ is induced by the following.
Let $[R,f]$ in $\hat T(X)$.
For any compact neighborhood $V$ of the set of nodes of $R$ and any $\varepsilon >0$, a neighborhood $U_{V,\varepsilon }$ of $[R,f]$ is defined by the set of $[S,g]$ in $\hat T(X)$ such that there is a deformation $h:S\rightarrow R$ which is $(1+\varepsilon )$-quasiconformal on $h^{-1}(R-V)$ such that $f$ is homotopic to $h\circ g$.

\subsection{The end points of Jenkins-Strebel rays}
We consider the end points of Jenkins-Strebel rays.
In the following discussion, we use the detailed description in \S 4.1 of \cite{HerSch07}.
Let $r$ be a Jenkins-Strebel ray on $T(X)$ from $r(0)=[Y,f]$ determined by a Jenkins-Strebel differential $q$ with unit norm on $Y$.
All components of $Y-\Gamma _q$ are represented by rectangles $C_1,\cdots ,C_k$ with identifications of vertical edges of them in $q$-coordinates.
Let $m_1,\cdots ,m_k$ be moduli of $C_1,\cdots ,C_k$ respectively.
We cut off each rectangle in the half height, and the resulting half rectangle $C_j^l$ is mapped conformally to the annulus $A_j^l(0)=\{ e^{-m_j\pi }\leq |z|<1\}$ for any $j=1,\cdots ,k$ and $l=1,2$.
Then, we can assume that the original surface $Y$ is constructed by $\{\overline{A_j^l(0)}\}_{j=1,\cdots ,k}^{l=1,2}$ with gluing mappings which are determined naturally.
Let $r(t)=[Y_t,g_t\circ f]$ be the representation of $r$ for any $t\geq 0$.
The Teichm\"uller mapping $g_t$ is represented by $z=re^{i\theta }\mapsto r^{e^{2t}}e^{i\theta }$ on each $A_j^l(0)$.
We set $A_j^l(t)=\{e^{-e^{2t}m_j\pi }\leq |z|<1\}$ for any $j=1,\cdots ,k$, $l=1,2$, and $t\geq 0$, then $Y_t$ is constructed by them as in the case of $t=0$.
In this representation, we can set $A_j^l(\infty )$ as the unit disk $\mathbb D=\{|z|<1\}$ for any $j=1,\cdots ,k$ and $l=1,2$.
We obtain the Riemann surface with nodes $Y_{\infty }$ by $\{\overline{A_j^l(\infty )}\}_{j=1,\cdots ,k}^{l=1,2}$ with the similar gluing mappings as in the case of $t\geq 0$.
The deformation $g_{\infty }:Y\rightarrow Y_{\infty }$ is obtained by $z=re^{i\theta }\mapsto h_j(r)e^{i\theta }$ on $A_j^l(\infty )$ where $h_j:[e^{-m_j\pi },1)\rightarrow [0,1)$ is an arbitrary monotonously increasing diffeomorphism for any $j=1,\cdots ,k$ and $l=1,2$.
The homotopy class of $g_{\infty }$ is independent of the choices of $h_j$ for any $j=1,\cdots ,k$.

\begin{prop}{\rm (cf. \cite{HerSch07})}
The Jenkins-Strebel ray $r(t)=[Y_t,g_t\circ f]$ on $T(X)$ converges to a point $r(\infty )=[Y_{\infty },g_{\infty }\circ f]$ in $\hat T(X)$ as $t\rightarrow \infty $.
\end{prop}

Suppose that $r$, $r'$ are similar Jenkins-Strebel rays on $T(X)$ from $r(0)=[Y,f]$, $r'(0)=[Y',f']$ determined by Jenkins-Strebel differentials $q$, $q'$ with unit norm on $Y$, $Y'$ respectively.
Let $\gamma _1,\cdots ,\gamma _k$ be as in the definition of ``similar" in \S \ref{geodesic}.
There is a homeomorphism $\alpha :X-f^{-1}(\Gamma _q)\rightarrow X-f'^{-1}(\Gamma _{q'})$ which is homotopic to the identity such that the mapping $f'\circ \alpha \circ f^{-1}$ maps the core curves of the annuli corresponding to $f(\gamma _j)$ to the core curves of the annuli corresponding to $f'(\gamma _j)$ for any $j=1,\cdots ,k$.
We set $r(\infty )=[Y_{\infty },g_{\infty }\circ f]$, $r'(\infty )=[Y'_{\infty },g'_{\infty }\circ f']$ and let $\{Y_{\infty ,\lambda }\}_{\lambda =1,\cdots ,\Lambda }$, $\{Y'_{\infty ,\lambda }\}_{\lambda =1,\cdots ,\Lambda }$ be the components of $Y_{\infty }-\{$nodes of $Y_{\infty }\}$, $Y'_{\infty }-\{$nodes of $Y'_{\infty }\}$ respectively, such that $(g'_{\infty }\circ f')\circ \alpha \circ (g_{\infty }\circ f)^{-1}(Y_{\infty ,\lambda})=Y'_{\infty ,\lambda}$ for any $\lambda =1,\cdots ,\Lambda $.
We define the Teichm\"uller distance between $r(\infty )$, $r'(\infty )$ by
	\begin{center}
	$\displaystyle d_{\hat T(X)}(r(\infty ),r'(\infty ))=\max _{\lambda =1,\cdots ,\Lambda }\frac{1}{2}\log \inf K(h _{\lambda })$,
	\end{center}
where the infimum ranges over all quasiconformal mappings $h_{\lambda }:Y_{\infty ,\lambda }\rightarrow Y'_{\infty ,\lambda }$ such that $h_{\lambda }$ is homotopic to $(g'_{\infty }\circ f')\circ \alpha \circ (g_{\infty }\circ f)^{-1}$.
This definition means that the distance between end points of Jenkins-Strebel rays is the maximum of the distance between the corresponding points of the Teichm\"uller space of each component of the end points of the rays.

\section{Proof of Theorem}
We recall our main theorem.

\setcounter{section}{1}
\begin{thm*}
For any two Jenkins-Strebel rays $r$, $r'$,
	\begin{eqnarray}
	(\ref{eq})
	&&\lim _{t\rightarrow \infty }d_{T(X)}(r(t),r'(t))=\nonumber \\
	&&\left\{\nonumber
	\begin{array}{l}
	\max \left\{\displaystyle \frac{1}{2}\log \max _{j=1,\cdots ,k}\left\{\frac{m'_j}{m_j},\frac{m_j}{m'_j}\right\}, d_{\hat T(X)}(r(\infty ),r'(\infty ))\right\}\\
	\hfill (${\rm if }$r$, $r'$ {\rm are similar}$)\\
	+\infty \ (${\rm otherwise}$) 
	\end{array}
	\right.
	\end{eqnarray}
\end{thm*}
\setcounter{section}{3}

We use the following lemma.
\begin{lemma} \label{lemma}
Let $R$, $R'$ be Riemann surfaces with nodes and $f:R\rightarrow R'$ be a $K$-quasiconformal Teichm\"uller mapping.
This means that $f$ is a homeomorphism, each restricted mapping of $f$ which maps a component of $R-\{$nodes of $R\}$ onto a component of $R'-\{$nodes of $R'\}$ is a Teichm\"uller mapping, and the maximum of maximal dilatations of such mappings is $K$.
Then, for any sufficiently small $\varepsilon >0$, there exists the $(K+O(\varepsilon ))$-quasiconformal mapping $g:R\rightarrow R'$ such that $g$ is conformal on a neighborhood of the set of nodes of $R$, and is homotopic to $f$.
\end{lemma}

The lemma is proved in the paper of \cite{FarMas10}, however we give a new proof of the latter part of their proof.

\begin{proof}[Proof of Lemma \ref{lemma}]
Let $\mu $ be the Beltrami coefficient of $f$.
For any $\varepsilon >0$, we consider a new Beltrami coefficient
	\begin{eqnarray} \nonumber
	\mu _\varepsilon =\left\{
	\begin{array}{l}
	0\ (0<|z|<\varepsilon )\\
	\mu \ ($otherwise$)
	\end{array}
	\right.
	\end{eqnarray}
on $R$, where $z$ is each local coordinate near nodes of $R$ and the domain $\{|z|<\varepsilon \}$ represents a neighborhood of nodes $z=0$.
Then, there exist a Riemann surface with nodes $R_{\varepsilon }$ and a $K$-quasiconformal mapping $f_{\varepsilon }:R\rightarrow R_{\varepsilon }$ such that $f_{\varepsilon }$ is conformal on the neighborhood of nodes of $R$.
For sufficiently small $\varepsilon $, $f_{\varepsilon }$ is close to $f$ and $R_{\varepsilon }$ is close to $R'$.
The mapping $f\circ f_{\varepsilon }^{-1}:R_{\varepsilon }\rightarrow R'$ is $K$-quasiconformal in a small neighborhood of nodes of $R_{\varepsilon }$ and is conformal on the outside of the neighborhood.
We use local coordinates such that nodes of $R_{\varepsilon }$ and $R'$ correspond to $0$, and $f\circ f_{\varepsilon }^{-1}(0)=0$.
Let $p_1,\cdots ,p_k$ be all nodes of $R_{\varepsilon }$.
For any $j=1,\cdots, k$, we take small disks $N_j^1=N_j^2=\{|z|<\delta \}$ about $p_j$ in $R_{\varepsilon }$ where $f\circ f_{\varepsilon }^{-1}$ is $K$-quasiconformal in $\overline{N_j^1}\cup \overline{N_j^2}$.
For any $j=1,\cdots ,k$ and $l=1,2$, the image $f\circ f_{\varepsilon }^{-1}(N_j^l)$ in $R'$ is mapped to the disk $\{|z|<\delta \}$ by a conformal mapping $f_{j,\varepsilon }^l$ such that $f_{j,\varepsilon }^l(0)=0$ and $f_{j,\varepsilon }^l(\delta )=\delta $.
We denote simply by $F_{\varepsilon }:=f_{j,\varepsilon }^l\circ f\circ f_{\varepsilon }^{-1}$.
The family of $K$-quasiconformal mappings $\{F_{\varepsilon }\}$ is normal, then we can assume that $F_{\varepsilon }$ converges to a $K$-quasiconformal mapping $F_0$ uniformly on any compact set of $\{|z|<\delta \}$ as $\varepsilon \rightarrow 0$.
However, any point of $\{0<|z|<\delta \}$ is a holomorphic point of $F_{\varepsilon }$ for sufficiently small $\varepsilon $, then $F_0$ is holomorphic in $\{0<|z|<\delta \}$.
Since $F_0$ fixes $0$ and $\delta $, we can see that $F_0$ is an automorphism on $\{|z|<\delta \}$ and then it is the identity.

Now, we rescale $\{|z|<\delta \}$ to $\mathbb D$, and assume that $\{F_{\varepsilon }\}$ as mappings of $\mathbb D$ onto itself.
We set a conformal mapping $\phi (z):=(z-i)/(z+i)$ of $\mathbb H=\{z\in \mathbb C\ |\ {\rm Im}z>0\}$ onto $\mathbb D$.
We consider mappings $f_{\varepsilon }:=\phi ^{-1}\circ F_{\varepsilon }\circ \phi$ of $\mathbb H$ onto itself.
\begin{lemma} \label{qs}
We have
	\begin{center}
	$\displaystyle \lim _{\varepsilon \rightarrow 0}\sup _{x,t}\frac{f_{\varepsilon }(x+t)-f_{\varepsilon }(x)}{f_{\varepsilon }(x)-f_{\varepsilon }(x-t)}=1$,
	\end{center}
where the supremum ranges over all $x,t\in \mathbb R$ such that $t\not =0$.
\end{lemma}

\begin{proof}[Proof of Lemma \ref{qs}]
We notice that $\frac{f_{\varepsilon }(x+t)-f_{\varepsilon }(x)}{f_{\varepsilon }(x)-f_{\varepsilon }(x-t)}$ is positive, and its reciprocal is the case of $-t$.
Then, it suffice to show that for any $t>0$.
Let
	\begin{center}
	$\displaystyle (z_1,z_2,z_3,z_4)=\frac{z_1-z_2}{z_1-z_3}\cdot \frac{z_3-z_4}{z_2-z_4}$
	\end{center}
be a cross ratio for any $z_1,z_2,z_3,z_4\in \hat {\mathbb C}$.
We write $\phi (x)=e^{i\theta }$, $\phi (x+t)=e^{i(\theta +\varphi )}$ and $\phi (x-t)=e^{i(\theta -\psi )}$ where $0\leq \theta <2\pi$ and $\varphi ,\psi >0$.
Since all M\"obius transformation preserve cross ratios,
	\begin{eqnarray}
	\frac{f_{\varepsilon }(x+t)-f_{\varepsilon }(x)}{f_{\varepsilon }(x)-f_{\varepsilon }(x-t)}&=&-(f_{\varepsilon }(x),f_{\varepsilon }(x+t),f_{\varepsilon }(x-t),\infty )\nonumber \\
	&=&-(F_{\varepsilon }\circ \phi (x),F_{\varepsilon }\circ \phi (x+t),F_{\varepsilon }\circ \phi (x-t),1)\nonumber \\
	&=&\left|\frac{F_{\varepsilon }(e^{i(\theta +\varphi )})-F_{\varepsilon }(e^{i\theta })}{F_{\varepsilon }(e^{i\theta })-F_{\varepsilon }(e^{i(\theta -\psi )})}\cdot \frac{F_{\varepsilon }(e^{i(\theta -\psi )})-1}{F_{\varepsilon }(e^{i(\theta +\varphi )})-1}\right|. \label{term}
	\end{eqnarray}
We set $z=e^{iy}$.
By $\log F_{\varepsilon }(e^{iy})=i(\arg F_{\varepsilon }(e^{iy})+2n\pi )$, we have
	\begin{center}
	$\displaystyle \frac{d\arg F_{\varepsilon }(e^{iy})}{dy}=\frac{z\frac{dF_{\varepsilon }(z)}{dz}}{F_{\varepsilon }(z)}$.
	\end{center}
Since $F_{\varepsilon }(z)$ and $\frac{dF_{\varepsilon }(z)}{dz}$ converge to $z$ and $1$ uniformly on $\partial \mathbb D$ respectively, then
	\begin{eqnarray}
	\sup _{0\leq y<2\pi }\left|\frac{d\arg F_{\varepsilon }(e^{iy})}{dy}-1\right|&=&\sup \left|z\frac{dF_{\varepsilon }(z)}{dz}-F_{\varepsilon }(z)\right|\nonumber \\
	&\leq &\sup \left(\left|\frac{dF_{\varepsilon }(z)}{dz}-1\right|+\left|z-F_{\varepsilon }(z)\right|\right)\rightarrow 0\nonumber
	\end{eqnarray}
as $\varepsilon \rightarrow 0$.
For any $0<E<1$, we take sufficiently small $\varepsilon $ such that
	\begin{center}
	$\displaystyle \left|\frac{d\arg F_{\varepsilon }(e^{iy})}{dy}-1\right|<E$
	\end{center}
holds for any $0\leq y<2\pi $.
Now, we calculate and estimate each term of (\ref{term}).
	\begin{eqnarray}
	|F_{\varepsilon }(e^{i(\theta +\varphi )})-F_{\varepsilon }(e^{i\theta })|&=&2\sin \frac{\arg F_{\varepsilon }(e^{i(\theta +\varphi )})-\arg F_{\varepsilon }(e^{i\theta })}{2}\nonumber \\
	&=&2\sin \frac{\int _{\theta }^{\theta +\varphi }\frac{d\arg F_{\varepsilon }(e^{iy})}{dy}dy}{2}\nonumber \\
	&<&2\sin \frac{(1+E)\varphi }{2}\nonumber \\
	&=&2\sin \frac{(1+E)(\arg \phi (x+t)-\arg \phi (x))}{2}\nonumber \\
	&=&2\sin \frac{(1+E)\int _x^{x+t}\frac{d\arg \phi (y)}{dy}dy}{2}\nonumber \\
	&=&2\sin \frac{(1+E)\int _x^{x+t}\frac{\frac{d\phi (y)}{dy}}{i\phi (y)}dy}{2}\nonumber \\
	&=&2\sin \frac{(1+E)\int _x^{x+t}\frac{2}{1+y^2}dy}{2}\nonumber \\
	&=&2\sin \{(1+E)(\arctan (x+t)-\arctan x)\},\nonumber
	\end{eqnarray}
and similarly,
	\begin{center}
	$\displaystyle |F_{\varepsilon }(e^{i(\theta +\varphi )})-F_{\varepsilon }(e^{i\theta })|>2\sin \{(1-E)(\arctan (x+t)-\arctan x)\}$.
	\end{center}
For any $0<\alpha \leq \pi /2$,
	\begin{eqnarray}
	\left|\frac{\sin \{(1\pm E)\alpha \}}{\sin \alpha }-1\right|&=&\left|\frac{\sin \alpha \cos (E\alpha )\pm \cos \alpha \sin (E\alpha )}{\sin \alpha }-1\right|\nonumber \\
	&=&\left|\cos (E\alpha )\pm \cos \alpha \frac{\sin (E\alpha )}{\sin \alpha }-1\right|\nonumber \\
	&\leq &|\cos (E\alpha )-1|+\left|\frac{\sin (E\alpha )}{\sin \alpha }\right|\nonumber \\
	&=&|\cos (E\alpha )-1|+\left|\frac{\sin (E\alpha )}{E\alpha }\right|\left|\frac{\alpha }{\sin \alpha }\right|E\nonumber \\
	&\leq &E\alpha +\frac{\pi }{2}E\leq \pi E\rightarrow 0\nonumber
	\end{eqnarray}
as $E\rightarrow 0$.
This means that we can write $\sin \{(1\pm E)\alpha \}=(1+O(E))\sin \alpha $.
Therefore,
	\begin{eqnarray}
	&&2\sin \{(1\pm E)(\arctan (x+t)-\arctan x)\}\nonumber \\
	&=&2(1+O(E))\sin (\arctan (x+t)-\arctan x).\nonumber
	\end{eqnarray}
We conclude that
	\begin{center}
	$|F_{\varepsilon }(e^{i(\theta +\varphi )})-F_{\varepsilon }(e^{i\theta })|<2(1+O(E))\sin (\arctan (x+t)-\arctan x)$
	\end{center}
and
	\begin{center}
	$|F_{\varepsilon }(e^{i(\theta +\varphi )})-F_{\varepsilon }(e^{i\theta })|>2(1+O(E))\sin (\arctan (x+t)-\arctan x)$.
	\end{center}
Also we have similar estimates for
	\begin{center}
	$|F_{\varepsilon }(e^{i\theta })-F_{\varepsilon }(e^{i(\theta -\psi )})|$, $|F_{\varepsilon }(e^{i(\theta -\psi )})-1|$, and $|F_{\varepsilon }(e^{i(\theta +\varphi )})-1|$.
	\end{center}
Finally, we can see that
	\begin{eqnarray}
	&&\frac{f_{\varepsilon }(x+t)-f_{\varepsilon }(x)}{f_{\varepsilon }(x)-f_{\varepsilon }(x-t)}\nonumber \\
	&<&(1+O(E))\frac{\sin (\arctan (x+t)-\arctan x)\sin (\frac{\pi }{2}-\arctan (x-t))}{\sin (\arctan x-\arctan (x-t))\sin (\frac{\pi }{2}-\arctan (x+t))}\nonumber \\
	&=&(1+O(E))\frac{\frac{t}{\sqrt{1+(x+t)^2}\sqrt{1+x^2}}\frac{1}{\sqrt{1+(x-t)^2}}}{\frac{t}{\sqrt{1+(x-t)^2}\sqrt{1+x^2}}\frac{1}{\sqrt{1+(x+t)^2}}}\nonumber \\
	&=&1+O(E).\nonumber
	\end{eqnarray}
The lower estimate is similar.
\end{proof}

Lemma \ref{qs} implies that the mapping $f_{j,\varepsilon }^l\circ f\circ f_{\varepsilon }^{-1}$ is $(1+O(\varepsilon ))$-quasisymmetric on the circle $\partial N_j^l=\{|z|=\delta \}$ for sufficiently small $\varepsilon $.
We apply Lemma 5.1 in \cite{Gupta11}.
Then, there exists a mapping $\eta _{j,\varepsilon }^l:N_j^l\rightarrow \{|z|<\delta \}$ which is $(1+O(\varepsilon ))$-quasiconformal, $\eta _{j,\varepsilon }^l|_{\partial N_j^l}=f_{j,\varepsilon }^l\circ f\circ f_{\varepsilon }^{-1}|_{\partial N_j^l}$, and $\eta _{j,\varepsilon }^l$ is the identity in a sufficiently small neighborhood of $0$. 
The mapping $(f_{j,\varepsilon }^l)^{-1}\circ \eta _{j,\varepsilon }^l:N_j^l\rightarrow f\circ f_{\varepsilon }^{-1}(N_j^l)$ is $(1+O(\varepsilon ))$-quasiconformal and satisfies $(f_{j,\varepsilon }^l)^{-1}\circ \eta _{j,\varepsilon }^l|_{\partial N_j^l}=f\circ f_{\varepsilon }^{-1}|_{\partial N_j^l}$ and $(f_{j,\varepsilon }^l)^{-1}\circ \eta _{j,\varepsilon }^l(0)=0$.
We consider the mapping of $R_{\varepsilon }$ onto $R'$ which is $(f_{j,\varepsilon }^l)^{-1}\circ \eta _{j,\varepsilon }^l$ in $N_j^l$ for any $j=1,\cdots ,k$ and $l=1,2$, and is $f\circ f_{\varepsilon }^{-1}$ on $R_{\varepsilon }-\bigcup _{j=1,\cdots ,k}^{l=1,2}N_j^l$.
This mapping is $(1+O(\varepsilon ))$-quasiconformal and is clearly homotopic to $f\circ f_{\varepsilon }^{-1}$.
Therefore, we conclude that there exists a $(1+O(\varepsilon ))$-quasiconformal Teichm\"uller mapping $h_{\varepsilon }:R_{\varepsilon }\rightarrow R'$ which is homotopic to $f\circ f_{\varepsilon }^{-1}$.

We deform $h_{\varepsilon }$ to a $(1+O(\varepsilon ))$-quasiconformal mapping which is conformal on a neighborhood of nodes of $R_{\varepsilon }$.
We use $\varepsilon '<1$ instead of $O(\varepsilon )$.
Let $q$ be the holomorphic quadratic differential on $R_{\varepsilon }-\{$nodes of $R_{\varepsilon }\}$ which is corresponding to the Teichm\"uller mapping $h_{\varepsilon }$.
Let $z=x+iy$ be any $q$-coordinate, then $h_{\varepsilon }$ is represented by $z\mapsto x+i(1+\varepsilon ')y$.
We consider the set $D_{\varepsilon }=\{ -\varepsilon '\leq x\leq \varepsilon ', 0\leq y\leq \varepsilon '\}-\{ 0\}$ where $0$ corresponds to a node of $R_{\varepsilon }$.
Let $H_{\varepsilon }$ be a mapping on the half set $\{ 0\leq x,y\leq \varepsilon '\}-\{ 0\}$ which is defined the following:
	\begin{eqnarray} \nonumber
	H_{\varepsilon }(z)=\left\{
	\begin{array}{ll}
	z &(0\leq x,y\leq \varepsilon '^2, z\not =0)\\
	x+i\frac{y-\varepsilon '^3}{1-\varepsilon '} &(0\leq x\leq \varepsilon '^2, \varepsilon '^2\leq y\leq \varepsilon ')\\
	x+i\frac{x+1-\varepsilon '-\varepsilon '^2}{1-\varepsilon '}y &(\varepsilon '^2\leq x\leq \varepsilon ', 0\leq y\leq \varepsilon '^2)\\
	x+i\frac{(1-\varepsilon '^2-\varepsilon '(x+1-\varepsilon '-\varepsilon '^2))y+\varepsilon '^2(x-\varepsilon ')}{(1-\varepsilon ')^2} &(\varepsilon '^2\leq x,y\leq \varepsilon ')
	\end{array}
	\right.
	\end{eqnarray}

An easy calculation shows that $H_{\varepsilon }$ is $(1+O(\varepsilon ))$-quasiconformal.
We extend $H_{\varepsilon }$ symmetrically to $D_{\varepsilon }$.
We assume that $z=0$ is a critical point of $q$ of order $n\geq -1$.
We consider the $n+2$ copies of $D_{\varepsilon }$ around $0$ whose horizontal segments of the boundary of $D_{\varepsilon }$ lie on the horizontal trajectories of $q$ which tend to $0$, and vertical segments of the boundary of $D_{\varepsilon }$ are joined adjacent to other $D_{\varepsilon }$ (Figure \ref{nbd1}).
	\begin{figure}[!ht]
	\begin{center}
	\includegraphics[keepaspectratio, scale=0.75]
	{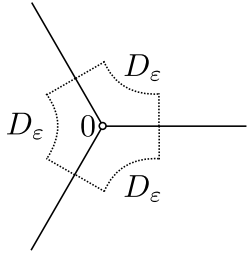}
	\caption{The gluing of $D_{\varepsilon }$ in the case of $n=1$}
	\label{nbd1}
	\end{center}
	\end{figure}

Finally, we extend the mapping $H_{\varepsilon }$ as the original mapping $h_{\varepsilon }$ on the outside of all $D_{\varepsilon }$, i.e., it is of the form $z\mapsto x+i(1+\varepsilon ')y$.
Then the mapping $H_{\varepsilon }$ is defined on the whole surface $R_{\varepsilon }$.
The quasiconformal mapping $H_{\varepsilon }\circ h_{\varepsilon }^{-1}:R'\rightarrow R'$ tends to the identity as $\varepsilon \rightarrow 0$.
Hence, for sufficiently small $\varepsilon $, the quasiconformal mapping $H_{\varepsilon }\circ h_{\varepsilon }^{-1}$ is homotopic to the identity because each mapping class group on components of $R'-\{ $nodes of $R'\}$ is discrete.
We conclude that the composition $H_{\varepsilon }\circ f_{\varepsilon }$ is homotopic to $f$, and it is our desired mapping $g$.
\end{proof}

\begin{proof}[Proof of Theorem \ref{main}]
If $r$, $r'$ are not similar, the result is already known.
We can see it in \cite{Ivanov01}, \cite{LenMas10} and also \cite{Amano14}.

Let $r$, $r'$ be similar Jenkins-Strebel rays on $T(X)$ from $r(0)=[Y,f]$, $r(0)'=[Y',f']$ determined by Jenkins-Strebel differentials $q$, $q'$ with unit norm on $Y$, $Y'$ respectively.
By definition, there exist mutually disjoint simple closed curves $\gamma _1,\cdots ,\gamma _k$ on $X$ such that the set of homotopy classes of core curves of the annuli corresponding to $q$, $q'$ are represented by $f(\gamma _1),\cdots ,f(\gamma _k)$ on $Y$ and $f'(\gamma _1),\cdots , f'(\gamma _k)$ on $Y'$ respectively.
Moreover, there is a homeomorphism $\alpha :X-f^{-1}(\Gamma _q)\rightarrow X-f'^{-1}(\Gamma _{q'})$ which is homotopic to the identity such that the mapping $f'\circ \alpha \circ f^{-1}$ maps the core curves of the annuli corresponding to $f(\gamma _j)$ to the core curves of the annuli corresponding to $f'(\gamma _j)$ for any $j=1,\cdots ,k$.
We denote by $m_j$, $m'_j$ the moduli of the annuli on $Y$, $Y'$ with core curves homotopic to $f(\gamma _j)$, $f'(\gamma _j)$ respectively.
For any $t\geq 0$, we set $r(t)=[Y_t,g_t\circ f]$, $r'(t)=[Y'_t,g'_t\circ f']$ where $g_t:Y\rightarrow Y_t$, $g'_t:Y'\rightarrow Y'_t$ are Teichm\"uller mappings.
Let $r(\infty )=[Y_{\infty },g_{\infty }\circ f]$, $r'(\infty )=[Y'_{\infty },g'_{\infty }\circ f']$ be the end points of $r$, $r'$ in the augmented Teichm\"uller space $\hat T(X)$ respectively.
Let $\{Y_{\infty ,\lambda }\}_{\lambda =1,\cdots ,\Lambda }$, $\{Y'_{\infty ,\lambda }\}_{\lambda =1,\cdots ,\Lambda }$ be the components of $Y_{\infty }-\{$nodes of $Y_{\infty }\}$, $Y'_{\infty }-\{$nodes of $Y'_{\infty }\}$ respectively, such that $(g'_{\infty }\circ f')\circ \alpha \circ (g_{\infty }\circ f)^{-1}(Y_{\infty ,\lambda})=Y'_{\infty ,\lambda}$ for any $\lambda =1,\cdots ,\Lambda $.

First, we consider the upper estimate.
Let $h_{\lambda }:Y_{\infty ,\lambda }\rightarrow Y'_{\infty ,\lambda }$ be the Teichm\"uller mapping which is homotopic to $(g'_{\infty }\circ f')\circ \alpha \circ (g_{\infty }\circ f)^{-1}$, and we set the mapping $h:Y_{\infty }\rightarrow Y'_{\infty }$ constructed by $\{h_{\lambda }\}_{\lambda =1,\cdots ,\Lambda }$.
We set $K=\exp (2d_{\hat T(X)}(r(\infty ),r'(\infty )))=\max _{\lambda =1,\cdots ,\Lambda } K(h_{\lambda })$.
The Riemann surfaces with nodes $Y_{\infty }$, $Y'_{\infty }$ are represented by the unions of closed unit disks $\{\overline {A_j^l(\infty )}\}_{j=1,\cdots ,k}^{l=1,2}$, $\{\overline {A_j^{'l}(\infty )}\}_{j=1,\cdots ,k}^{l=1,2}$ respectively.
Let $h_j^l:=h|_{A_j^l(\infty )}$ be the restriction to $A_j^l(\infty )$ of $h$ for any $j=1,\cdots ,k$ and $l=1,2$, then there is $\lambda $ such that $h_j^l$ is equal to $h_{\lambda }|_{A_j^l(\infty )}$.
We apply Lemma \ref{lemma} to the mapping $h:Y_{\infty }\rightarrow Y'_{\infty }$.
Hence, for any $j=1,\cdots ,k$ and $l=1,2$, we assume that $h_j^l$ is $(K+O(\varepsilon ))$-quasiconformal such that it is conformal in a neighborhood of $0$ in $A_j^l(\infty )$ and is homotopic to $(g'_{\infty }\circ f')\circ \alpha \circ (g_{\infty }\circ f)^{-1}$.
Now, we can use the idea of the proof of Theorem \ref{main} in \cite{Amano14}.
In this neighborhood, $h_j^l$ is represented by a power series, i.e., we can write $h_j^l(z)=c_j^lz+c_{j,2}^lz^2+\cdots =c_j^lz+\psi _j^l(z)$ where $c_j^l\not =0$, $-\pi <\arg c_j^1\leq \pi $ and $-\pi \leq \arg c_j^2<\pi $.
For any $j=1,\cdots ,k$, we set $M_j=\frac{m'_j}{m_j}$, and for any $t\geq 0$, $\delta _j(t)=e^{-e^{2t}m_j\pi }$, $\delta '_j(t) = e^{-e^{2t}m'_j\pi }$, then $\delta '_j(t)=\delta _j(t)^{M_j}$.
We only consider the case of $M_j>1$, so we fix such $j$.
For sufficiently small $\varepsilon $, again, we use $\varepsilon '<1$ instead of $O(\varepsilon )$.
We take $X_j$ as
	\begin{center}
	$\displaystyle X_j<\frac{\log \frac{\varepsilon '}{M_j+\varepsilon '-1}}{\log M_j}<0$.
	\end{center}
We take sufficiently large $t$ such that an inequality $\delta _j(t)^{M_j}<|c_j^l|\delta _j(t)^{M_j^{X_j}}$ holds, and set $\Delta _j(t)=\delta _j(t)^{M_j^{X_j}}$.
Also, we assume that a domain such that $h_j^l$ can be represented by the power series contains $\{|z|\leq 2\Delta _j(t)\}$.
We construct $F_{j,t}^l:A_j^l(t)\rightarrow h(A_j^l(t))-\{|z|<\delta '_j(t)\}$ by the following:
	\begin{eqnarray}
	F_{j,t}^l(z)=
	\left\{
	\begin{array}{lll}
	P_{j,t}^l(z)&(\delta _j(t)\leq |z|\leq \Delta _j(t))&{\rm (i)}\nonumber \\
	Q_{j,t}^l(z)&(\Delta _j(t)\leq |z|\leq 2\Delta _j(t))&{\rm (ii)}\nonumber \\
	h_j^l(z)&(2\Delta _j(t)\leq |z|<1)&{\rm (iii)}\nonumber
	\end{array}
	\right.
	\end{eqnarray}

(i) In $\delta _j(t)\leq |z|\leq \Delta _j(t)$, we set
	\begin{center}
	$P_{j,t}^l(z)=\Delta _j(t)^\frac{1-M_j}{1-M_j^{X_j}}\cdot {c_j^l}^{\frac{1}{1-M_j^{X_j}}+\frac{\log |z|}{\log \Delta _j(t)-\log \delta _j(t)}}\cdot |z|^{-\frac{1-M_j}{1-M_j^{X_j}}}\cdot z$
	\end{center}
which satisfies $P_{j,t}^l(z)=\delta _j(t)^{M_j-1}\cdot z$ on $|z|=\delta _j(t)$, $P_{j,t}^l(z)=c_j^lz$ on $|z|=\Delta _j(t)$.
The mapping $P_{j,t}^l$ is a quasiconformal mapping because it is conjugate to a one-to-one affine mapping by $\log z$.
The maximal dilatation of $P_{j,t}^l$ is
	\begin{center}
	$\displaystyle K(P_{j,t}^l)=\frac{\left|\frac{\log c_j^l}{2(M_j^{X_j}-1)\log \delta _j(t)}+\frac{\alpha _j}{2}+1\right|+\left|\frac{\log c_j^l}{2(M_j^{X_j}-1)\log \delta _j(t)}+\frac{\alpha _j}{2}\right|}{\left|\frac{\log c_j^l}{2(M_j^{X_j}-1)\log \delta _j(t)}+\frac{\alpha _j}{2}+1\right|-\left|\frac{\log c_j^l}{2(M_j^{X_j}-1)\log \delta _j(t)}+\frac{\alpha _j}{2}\right|}$,
	\end{center}
where $\alpha _j=-\frac{1-M_j}{1-M_j^{X_j}}$.
We see that
	\begin{center}
	$\displaystyle K(P_{j,t}^l)\rightarrow \frac{M_j-M_j^{X_j}}{1-M_j^{X_j}}<M_j+\varepsilon '$
	\end{center}
as $t\rightarrow \infty $.

(ii) In $\Delta _j(t)\leq |z|\leq 2\Delta _j(t)$, we set
	\begin{center}
	$Q_{j,t}^l(z)=c_j^lz+\phi _{\Delta _j(t)}(|z|)\psi _j^l(z)$,
	\end{center}
where $\phi _{\Delta _j(t)}:[\Delta _j(t),2\Delta _j(t)]\rightarrow [0,1]$ is defined by
	\begin{center}
	$\displaystyle \phi _{\Delta _j(t)}(|z|)=\frac{|z|}{\Delta _j(t)}-1$.
	\end{center}
Then $Q_{j,t}^l(z)=c_j^lz$ on $|z|=\Delta _j(t)$, $Q_{j,t}^l(z)=h_j^l(z)$ on $|z|=2\Delta _j(t)$.	
We consider the partial derivatives of $Q_{j,t}^l$,
	\begin{center}
	$\displaystyle \partial _{\bar{z}}Q_{j,t}^l=\frac{1}{2\Delta _j(t)}z^{\frac{1}{2}}\bar{z}^{-\frac{1}{2}}\psi _j^l(z)$,\\
	$\displaystyle \partial _zQ_{j,t}^l=c_j^l+\frac{1}{2\Delta _j(t)}z^{-\frac{1}{2}}\bar{z}^{\frac{1}{2}}\psi _j^l(z)+\phi _{\Delta (t)}(|z|)\frac{d\psi _j^l(z)}{dz}$.
	\end{center}
These are continuous in $\Delta _j(t)\leq |z|\leq 2\Delta _j(t)$.
There is $C>0$ such that $|\psi _j^l(z)|\leq C\Delta _j(t)^2$ for sufficiently large $t$.
We see that
	\begin{center}
	$\displaystyle \left|\frac{1}{2\Delta _j(t)}z^{\frac{1}{2}}\bar{z}^{-\frac{1}{2}}\psi _j^l(z)\right|=\left|\frac{1}{2\Delta _j(t)}z^{-\frac{1}{2}}\bar{z}^{\frac{1}{2}}\psi _j^l(z)\right|=\frac{|\psi _j^l(z)|}{2\Delta _j(t)}\leq \frac{C\Delta _j(t)}{2}\rightarrow 0$	
	\end{center}
and then $|\partial _{\bar{z}}Q_{j,t}^l|\rightarrow 0$, $|\partial _zQ_{j,t}^l|\rightarrow |c_j^l|\not =0$ as $t\rightarrow \infty $.
Hence, for sufficiently large $t$, $\jac Q_{j,t}^l=|\partial _zQ_{j,t}^l|^2-|\partial _{\bar{z}}Q_{j,t}^l|^2>0$, and we conclude that $Q_{j,t}^l$ is a local $C^1$-diffeomorphism.
We denote by $D$ the closed set whose fundamental group is $\pi _1(D)=\mathbb Z$ and its boundary components are $Q_{j,t}^l(\{ |z|=\Delta _j(t)\} )=\{ |w|=|c_j^l|\Delta _j(t)\}$ and $Q_{j,t}^l(\{ |z|=2\Delta _j(t)\} )=h_j^l(\{ |z|=2\Delta _j(t)\} )$.
Since $Q_{j,t}^l$ is a local $C^1$-diffeomorphism, we have $Q_{j,t}^l(\{ \Delta _j(t)\leq |z|\leq 2\Delta _j(t)\})=D$.
Furthermore, by the compactness of $\{ \Delta _j(t)\leq |z|\leq 2\Delta _j(t)\}$, $Q_{j,t}^l$ is proper.
Then we can regard the mapping $Q_{j,t}^l:\{ \Delta _j(t)\leq |z|\leq 2\Delta _j(t)\}\rightarrow D$ as a covering.
Let $Q_{j,t*}^l:\pi _1(\{ \Delta _j(t)\leq |z|\leq 2\Delta _j(t)\} )\rightarrow \pi _1(D)$ be the group homomorphism induced by $Q_{j,t}^l$.
We see that $Q_{j,t*}^l(\pi _1(\{ \Delta _j(t)\leq |z|\leq 2\Delta _j(t)\} ))=\mathbb Z\triangleleft \pi _1(D)$ because $Q_{j,t}^l(z)=c_j^lz$ on $|z|=\Delta _j(t)$.
Then, the covering $Q_{j,t}^l$ is regular, and its covering transformation group is $\mathbb Z/\mathbb Z=1$.
Therefore, we conclude that $Q_{j,t}^l$ is a $C^1$-diffeomorphism.
By the partial derivatives of $Q_{j,t}^l$, for sufficiently large $t$, it is a quasiconformal mapping and satisfies $K(Q_{j,t}^l)\rightarrow 1$ as $t\rightarrow \infty $.

(iii) In $2\Delta _j(t)\leq |z|<1$, $F_{j,t}^l(z)=h_j^l(z)$ and $K(h_j^l)\leq K$.

By the above discussions, for sufficiently large $t$, we obtain a quasiconformal mapping $F_{j,t}^l$ such that
	\begin{eqnarray}
	K(F_{j,t}^l)=\max \{ K(P_{j,t}^l),K(Q_{j,t}^l),K(h_j^l)\} &\rightarrow &\max \left\{\frac{M_j-M_j^{X_j}}{1-M_j^{X_j}},K(h_j^l)\right\}\nonumber \\
	&<&\max \left\{M_j,K\right\}+\varepsilon '\nonumber
	\end{eqnarray}
as $t\rightarrow \infty $.

In the cases of $M_j<1$, $M_j=1$, we also have
	\begin{center}
	$\displaystyle \lim _{t\rightarrow \infty }K(F_{j,t}^l)<\max \left\{\frac{1}{M_j},K\right\}+\varepsilon '$
	\end{center}
and
	\begin{center}
	$\displaystyle \lim _{t\rightarrow \infty }K(F_{j,t}^l)=K$
	\end{center}
by similar arguments.

Thus, for sufficiently large $t$, we can construct the quasiconformal mapping $F_t:Y_t\rightarrow Y'_t$ by gluing $\{F_{j,t}^l\} _{j=1,\cdots ,k}^{l=1,2}$.
We obtain the inequality
	\begin{center}
	$\displaystyle \lim _{t\rightarrow \infty }K(F_t)<\max \left\{\max _{j=1,\cdots ,k}\left\{\frac{m'_j}{m_j},\frac{m_j}{m'_j}\right\},K\right\}+\varepsilon '$.
	\end{center}
Next, we confirm that $F_t$ is homotopic to $(g'_t\circ f')\circ (g_t\circ f)^{-1}$.
In any case, each $h_j^l$ is homotopic to $(g'_t\circ f')\circ \alpha \circ (g_t\circ f)^{-1}$ in $\{2\Delta _j(t)<|z|<1\}$.
Each $Q_{j,t}^l$ satisfies $K(Q_{j,t}^l)\rightarrow 1$ as $t\rightarrow \infty $ and the domain $\{\Delta _j(t)<|z|<2\Delta _j(t)\}$ has the constant modulus for any $t$.
Finally, each $P_{j,t}^l$ produces a twist of angle $\arg c_j^l$ in $\{\delta _j(t)<|z|<\Delta _j(t)\}$ and satisfies $|\arg c_j^1+\arg c_j^2|<2\pi $.
Therefore, for sufficiently large $t$, the mapping $F_t$ is homotopic to $(g'_t\circ f')\circ \alpha \circ (g_t\circ f)^{-1}$.
Since $\alpha $ is homotopic to the identity on $X$, we are done.
We conclude that
	\begin{center}
	$\displaystyle \limsup _{t\rightarrow \infty }d_{T(X)}(r(t),r'(t))\leq \max \left\{\frac{1}{2}\log \max _{j=1,\cdots ,k}\left\{\frac{m'_j}{m_j},\frac{m_j}{m'_j}\right\}, d_{\hat T(X)}(r(\infty ),r'(\infty ))\right\}$.
	\end{center}

For the lower estimate, we can use the following inequality.
\begin{prop}{\rm (\cite{Amano14})}
We have
	\begin{center}
	$\displaystyle \liminf _{t\rightarrow \infty }d_{T(X)}(r(t),r'(t))\geq \frac{1}{2}\log \max _{j=1,\cdots ,k}\left\{\frac{m'_j}{m_j},\frac{m_j}{m'_j}\right\}$.
	\end{center}
\end{prop}

Furthermore, we use the following fact.

\begin{prop}{\rm (\cite{Masur75})}
We have
	\begin{center}
	$\displaystyle \liminf _{t\rightarrow \infty }d_{T(X)}(r(t),r'(t))\geq d_{\hat T(X)}(r(\infty ),r'(\infty ))$.
	\end{center}
\end{prop}

Combining the above two inequalities, we obtain the inequality
	\begin{center}
	$\displaystyle \liminf _{t\rightarrow \infty }d_{T(X)}(r(t),r'(t))\geq \max \left\{\frac{1}{2}\log \max _{j=1,\cdots ,k}\left\{\frac{m'_j}{m_j},\frac{m_j}{m'_j}\right\}, d_{\hat T(X)}(r(\infty ),r'(\infty ))\right\}$.
	\end{center}
\end{proof}

\setcounter{section}{1}
\begin{cor*}
If $r$, $r'$ are similar, the minimum value of the equation {\rm (\ref{eq})} when we shift the starting points of $r$, $r'$ is given by
	\begin{eqnarray}
	\max \left\{\frac{1}{2}\delta , d_{\hat T(X)}(r(\infty ),r'(\infty ))\right\},\nonumber
	\end{eqnarray}
where $\displaystyle \delta =\frac{1}{2}\log \max _{j=1,\cdots ,k}\frac{m'_j}{m_j}+\frac{1}{2}\log \max _{j=1,\cdots ,k}\frac{m_j}{m'_j}$.
\end{cor*}
\setcounter{section}{3}

\begin{proof}[Proof of Corollary \ref{cor}]
We see that
	\begin{center}
	$\displaystyle \frac{1}{2}\log \max _{j=1,\cdots ,k}\left\{\frac{m'_j}{m_j},\frac{m_j}{m'_j}\right\}\geq \frac{1}{2}\delta $.
	\end{center}
The values $\frac{1}{2}\delta $ and $d_{\hat T(X)}(r(\infty ),r'(\infty ))$ are invariant when we shift the starting points of the rays $r$, $r'$.
Hence, by Theorem \ref{main}, 
	\begin{center}
	$\displaystyle \lim _{t\rightarrow \infty }d_{T(X)}(r(t),r'(t+\alpha ))\geq \max \left\{\frac{1}{2}\delta , d_{\hat T(X)}(r(\infty ),r'(\infty ))\right\}$
	\end{center}
for any $\alpha \in \mathbb R$.
The equality holds if
	\begin{center}
	$\displaystyle \alpha =\frac{1}{4}\log \frac{\displaystyle \max _{j=1,\cdots ,k}\frac{m_j}{m'_j}}{\displaystyle \max _{j=1,\cdots ,k}\frac{m'_j}{m_j}}$.
	\end{center}
Indeed, we calculate that
	\begin{eqnarray}
	\max _{j=1,\cdots ,k}\frac{e^{2\alpha }m'_j}{m_j}=\max _{j=1,\cdots ,k}\left\{\frac{\displaystyle \sqrt{\max _{j=1,\cdots ,k}\frac{\mathstrut m_j}{m'_j}}\cdot m'_j}{\displaystyle \sqrt{\max _{j=1,\cdots ,k}\frac{m'_j}{m_j}}\cdot m_j}\right\}&=&\sqrt{\max _{j=1,\cdots ,k}\frac{m'_j}{m_j}}\cdot \sqrt{\max _{j=1,\cdots ,k}\frac{\mathstrut m_j}{m'_j}}\nonumber \\
	&=&\max _{j=1,\cdots ,k}\frac{m_j}{e^{2\alpha }m'_j}.\nonumber
	\end{eqnarray}
Therefore, we conclude that
	\begin{eqnarray}
	&&\lim _{t\rightarrow \infty }d_{T(X)}(r(t),r'(t+\alpha ))\nonumber \\
	&=&\max \left\{\frac{1}{2}\log \max _{j=1,\cdots ,k}\left\{\frac{e^{2\alpha }m'_j}{m_j},\frac{m_j}{e^{2\alpha }m'_j}\right\}, d_{\hat T(X)}(r(\infty ),r'(\infty ))\right\}\nonumber \\
	&=&\max \left\{\frac{1}{2}\left(\frac{1}{2}\log \max _{j=1,\cdots ,k}\frac{m'_j}{m_j}+\frac{1}{2}\log \max _{j=1,\cdots ,k}\frac{m_j}{m'_j}\right), d_{\hat T(X)}(r(\infty ),r'(\infty ))\right\}\nonumber \\
	&=&\max \left\{\frac{1}{2}\delta , d_{\hat T(X)}(r(\infty ),r'(\infty ))\right\}.\nonumber
	\end{eqnarray}
\end{proof}

\section*{Acknowledgements}
I would like to express the deepest appreciation to Professor Hiroshige Shiga for his insightful comments and suggestions.
This work is supported by Global COE Program ``Computationism as a Foundation for the Sciences".

\bibliographystyle{alpha}
\bibliography{references}

\end{document}